\documentclass{amsart}

\usepackage{latexsym,amsfonts}
\usepackage{amssymb, amsbsy}
\usepackage{amsmath}
\usepackage{amsthm}

\newtheorem{lemma}{Lemma}[section]
\newtheorem{theorem}[lemma]{Theorem}
\newtheorem{cor}[lemma]{Corollary}
\newtheorem{proposition}[lemma]{Proposition}
\newtheorem{conj}[lemma]{Conjecture}

\theoremstyle{remark}

\theoremstyle{definition}
\newtheorem{example}[lemma]{Example}

\newcommand{\Irr}[1]{\mathrm{Irr}(#1)}
\newcommand{\Sym}[1]{\mathrm{Sym}(#1)}
\newcommand{\Alt}[1]{\mathrm{Alt}(#1)}

\newcommand{\vet}[1]{\overrightarrow{#1}}
\newcommand{\Cyc}[1]{\mathrm{Cyc}(#1)}
\newcommand{\PSL}{\mathrm{PSL}}
\newcommand{\PSU}{\mathrm{PSU}}
\newcommand{\SL}{\mathrm{SL}}
\newcommand{\SU}{\mathrm{SU}}

\def\F{\mathbb{F}}
\def\Z{\mathbb{Z}}

\def\Q{\mathbb{Q}}

\begin{document}

\title[Irreducible $p$-constant characters of finite reflection groups]{Irreducible 
$p$-constant characters of \\finite reflection groups}

\author{Marco Antonio Pellegrini}
\address{Dipartimento di Matematica e Fisica, Universit\`a Cattolica del Sacro Cuore, Via
Musei 41,
25121 Brescia, Italy}
\email{marcoantonio.pellegrini@unicatt.it}



\thanks{This work was supported by the ``National Group for Algebraic and Geometric Structures, and their Applications'' (GNSAGA-INDAM)}

\begin{abstract}
A complex irreducible character of a finite group $G$ is said to be $p$-constant, for some
prime $p$ dividing the order of $G$, if it takes constant value at the set of $p$-singular
elements of $G$. In this paper we classify irreducible $p$-constant characters for finite 
reflection groups, nilpotent groups and complete monomial groups.
We also propose some conjectures about the structure of the groups admitting such characters.
\end{abstract}

\maketitle

\section{Introduction}
 
Let $G$ be a finite group. Fixing a prime $p$ that divides the order of $G$, we denote by $\varSigma_p(G)$ the set of the $p$-singular elements of $G$, i.e. the elements whose 
orders are divisible by $p$. 
We say that a (complex) irreducible character $\chi$ of $G$ is \emph{$p$-constant} if 
it takes
constant value at $\varSigma_p(G)$, that is, if there exists a constant $c_\chi$ such that
$\chi(g)=c_\chi$ for all $g\in \varSigma_p(G)$.
By \cite[Lemma 2.1]{PZ}, the constant $c_\chi$ must be a rational
integer. 
It is easy to see that the irreducible characters $\chi$ of $p$-defect zero  (i.e. whose degree 
equals the order of a Sylow
$p$-subgroup of $G$) are precisely the irreducible  $p$-constant characters with $c_\chi=0$.
In view of this, we can look at $p$-constant characters as a generalization of  $p$-defect zero 
characters (a different generalization to the reducible case was considered in \cite{PZ2}).
In this 
paper 
we will always consider the case $c_\chi\neq 0$: for future convenience, we denote by $\varUpsilon_p(G)$ the set of the irreducible $p$-constant characters of $G$ of non-zero $p$-defect.

The problem of classifying irreducible $p$-constant characters was firstly introduced in \cite{PZ}, 
where  finite quasi-simple groups have been studied. In particular,  the following result was obtained.

\begin{theorem}\cite[Theorem 1.5]{PZ}\label{simple}
Let $G$ be a finite simple group admitting an irreducible $p$-constant character $\chi$ for some prime $p$ dividing the order of $G$. Then, one of the following holds:
\begin{itemize}
\item[(i)] $c_\chi\in \{0,\pm 1\}$;
\item[(ii)] $G={\rm M}_{22}$, $p=3$, $\chi(1)=385$ and $c_\chi=-2$;
\item[(iii)] $G$ is a group of Lie type of characteristic $r\neq p$ with a non-cyclic Sylow $p$-subgroup. 
\end{itemize}
\end{theorem}

In the same paper, also the $p$-constant characters of the symmetric groups were classified. Starting from that result, in the present paper we study the other finite reflection groups as well as other notable classes of groups as nilpotent groups, direct and wreath products of groups. 

In Section \ref{Sec2} we prove in particular that any non-linear irreducible $p$-constant character of a finite nilpotent group must be of  $p$-defect zero.
In Section \ref{Sec3} we classify the finite reflection groups admitting 
$p$-constant characters  proving  the following results.

\begin{theorem}\label{th1}
Let $G$ be a finite reflection group. Then, for any prime $p$  dividing the order of $G$ and for any $\chi\in \varUpsilon_p(G)$ we have $c_\chi=\pm 1$.
\end{theorem}

\begin{theorem}\label{th2}
Let $G$ be a finite reflection group.
Then, $\varUpsilon_2(G)\neq \{1_G\}$  if and only if $G$ is a dihedral group of order $2n$ with $n$ odd. In this case, $\varUpsilon_2(G)$ consists of the two linear characters of $G$.
\end{theorem}

The classification of $p$-constant characters for groups of type ${\rm B}_n$ relies on a combinatorial version of Murnaghan-Nakayama rule given in \cite{Bn}. The same techniques used here can be also applied for studying endotrivial modules of such groups (as done for symmetric groups in \cite{end}).
In Section \ref{Sec4} we present some open problems and conjectures that we hope can attract the 
interest on this subject. For any odd prime $p$, we construct a Frobenius group admitting an irreducible $p$-constant character $\chi$ with $c_\chi=2$. We also classify the $p$-constant characters of $\PSL_3(q)$ and 
$\PSU_3(q)$, proving that case (iii) of Theorem \ref{simple} contains an infinite number of groups. It is worth of observe that by \cite[Theorem 6.1]{LM} case (iii)  may contain only groups whose Sylow $p$-subgroups have rank $\leq 2$, with the possible exceptions of $\PSL_5(q)$, $\PSU_5(q)$ and ${\rm E}_8(q)$.

\subsection*{Notations} 

To avoid misunderstanding with the reflection groups, we denote the cyclic 
group of order $n$ by $\Cyc{n}$. By $\Alt{n}$ and $\Sym{n}$ we mean, respectively, the alternating and 
the symmetric group of degree $n$.
Given a finite group $G$ and a positive integer $n$, $G^n$ denotes the direct product 
of $n$ copies of $G$.
If $p$ is a prime, $|G|_p$ denotes the order of a Sylow $p$-subgroup of $G$.

\section{Linear characters and nilpotent groups}\label{Sec2}

Let $G$ be a finite group. Clearly, the principal character $1_G$ of $G$ is $p$-constant for any prime $p$ dividing the order of $G$ (i.e., $1_G \in \varUpsilon_p(G)$ whenever $|G|_p>1$).
More in general, for linear characters we prove the following.

\begin{proposition}\label{linear}
Let $\chi$ be a linear character of a finite group $G$. Then, $\chi$ is $p$-constant for
some prime $p$ if, and only if, one of the following cases occurs:
\begin{itemize}
\item[(i)] $c_\chi=+1$ and $\varSigma_p(G)\subseteq \ker(\chi)$;
\item[(ii)] $c_\chi=-1$, $p=2$ and  $G=\ker(\chi) \rtimes \langle x \rangle$, where $\ker(\chi)$ has 
odd order and $x$ is an involution.
\end{itemize}
\end{proposition}

\begin{proof}
Since $\chi$ is linear, if $\chi$ is $p$-constant then  $c_\chi\in \Z$ and so $c_\chi=\pm 1$. Clearly, if $c_\chi=1$ then $\varSigma_p(G)\subseteq \ker(\chi)$. 
If $c_\chi=-1$ then $p=2$ and $\ker(\chi)$ has odd order. Working in $H=G/\ker(\chi)$, 
consider the linear character
$\bar \chi$ induced by $\chi$ on $H$. Since $\bar\chi$ is a faithful linear
character of $H$, then $H$ is cyclic, say $H=\langle \bar x\rangle$.
From  $\bar \chi(\bar x)=-1$ we get $\bar x^ 2\in \ker(\bar \chi)$ and so $|H|=2$.
Hence, $G=\ker(\chi) \rtimes \langle x \rangle$ for some involution $x$.
\end{proof}

We start our analysis considering finite $p$-groups and nilpotent groups. We first recall the following result.

\begin{lemma}\cite[Lemma 2.2]{PZ}\label{lm:center}
Let $G$ be a finite group and let $1_G\neq \chi\in \varUpsilon_p(G)$ for some prime $p$ 
dividing the order of $G$.  Then, one of the following holds:
\begin{itemize}
\item[(i)] $p$ does not divide $|Z(G)|$, $Z(G)\leq \ker(\chi)$ and the corresponding character 
$\bar \chi$ of $G/Z(G)$ 
is an irreducible $p$-constant character;
\item[(ii)] $p=2$, $|G|_2=2$ and $G=\ker(\chi)\times O_2(G)$.
\end{itemize} 
\end{lemma}

As consequence of the previous lemma we get the following.

\begin{cor}\label{p-group}
Let $p$ be a prime and let $G$ be a finite $p$-group.
Then $\varUpsilon_p(G)\neq \{1_G\}$  if and  only if $|G|=p=2$.
\end{cor}

\begin{proposition}\label{direct}
Let $H,K$ be two finite  groups and let $G=H\times K$.
Let $\chi\in \varUpsilon_p(G)$ for some prime
$p$ dividing the order of $G$. If  $p$ divides both $|H|$ and $|K|$, then 
$\chi=1_G$.
If $p$ divides $|H|$ but does not divide $|K|$, then $\chi=\psi\otimes 1_K$ for some $\psi\in \varUpsilon_p(H)$.
\end{proposition}

\begin{proof}
Suppose that $p$ divides the order of $H$. Since $G$ is a direct product of groups,
$\chi=\psi\otimes \varphi$ for some $\psi\in \Irr{H}$ and $\varphi\in \Irr{K}$. Fix an element
$h\in \varSigma_p(H)$:  for any $k\in K$, we have $hk\in \varSigma_p(G)$ and
$\chi(hk)=\psi(h)\varphi(k)=c_\chi\neq 0$. Hence, $\varphi(k)$ has the same value for all $k\in K$ and so $\varphi=1_K$,
i.e. $\chi=\psi \otimes 1_K$. It follows that for any $h\in \varSigma_p(H)$ we have
$\chi(h)=\psi(h)=c_\chi$, proving that $\psi$ is $p$-constant with $c_\psi=c_\chi\neq 0$.

Now, assume that $p$ also divides the order of $K$. Proceeding as  before, by
symmetry we obtain that $\chi=1_H\otimes 1_K=1_G$.
\end{proof}

\begin{cor}\label{pr}
Let $G=H\times K$, where $H$ is a finite $p$-group and $K$ is a finite $r$-group for some primes $p\neq r$.
If there exists $1_G\neq \chi \in \varUpsilon_p(G)$, then $p=|H|=2$ and $\chi=\epsilon \otimes 1_K$, where $\epsilon$ is the non-principal character of $H$.
\end{cor}

\begin{proof}
By Proposition \ref{direct}, $\chi=\psi\otimes 1_K$ for some  $\psi\in \varUpsilon_p(H)$. Now, since $H$ is a $p$-group and $\chi\neq 1_{G}$, by Corollary
\ref{p-group}, we have $p=|H|=2$ and $\psi=\epsilon$.
\end{proof} 

\begin{cor}
Let $G$ be a finite nilpotent group and let $\chi\in \varUpsilon_p(G)$. Then $\chi$ is linear and either
$\chi=1_G$ or $p=|G|_2=2$.
\end{cor}

\begin{proof}
The result follows from Corollary \ref{pr}, writing  $G$ as direct product  of its
Sylow $p_i$-subgroups.
\end{proof}

\section{Finite reflection groups}\label{Sec3}

In this section we consider finite groups generated by reflections (in real euclidean spaces). We will refer to \cite{H} for construction, properties and notation concerning this class of groups. We recall in particular that 
the finite reflection groups are precisely the finite Coxeter groups. Completely classified by
Coxeter himself, they are direct products of members of four infinite families (${\rm A}_n$, ${\rm 
B}_n$, 
${\rm D}_n$, ${\rm I}_2(n)$) and six `exceptional' groups (${\rm E}_6$, ${\rm E}_7$, ${\rm E}_8$, 
${\rm F}_4$, ${\rm H_3}$, ${\rm H}_4$). 
Therefore, by Proposition \ref{direct} we may restrict our attention to the `irreducible' finite reflection groups.

The $p$-constant characters of exceptional reflection groups can be classified using 
the character tables available, for instance, in \cite{GAP}. In the following proposition, $\chi_m$ denotes the $m$-th irreducible 
character, according to the ordering of \cite[CTblLib 1.2.2]{GAP}.

\begin{proposition}\label{ecc}
Let $G$ be a reflection group of type ${\rm E}_6, {\rm E}_7, {\rm E}_8, {\rm F}_4, {\rm H}_3$ 
or ${\rm H}_4$ and let $1_G\neq \chi_m \in \Irr{G}$.
Then $\chi_m\in \varUpsilon_p(G)$ if and only if one of the following cases occurs:
\begin{itemize}
\item[(i)] $G={\rm E}_6$, $p=5$ and $(m,\chi_m(1))\in \{(5,6), (13, 24), (23, 64), ( 25, 81)\}$;
\item[(ii)] $G={\rm E}_7$ and 
\begin{itemize}
\item[$p=5$:] $(m,\chi_m(1))\in \{(9, 56),( 11, 84), ( 18, 189), ( 22, 216)\}$ or
\item[$p=7$:] $(m,\chi_m(1))\in \{( 3, 15),( 6, 27),( 15, 120),(22,216), ( 28, 405),(
30,512)\}$;
\end{itemize}
\item[(iii)] $G={\rm E}_8$, $p=7$ and $(m,\chi_m(1))\in \{( 9, 50), (18, 300),( 36, 972),( 61, 3200),$ $( 62, 4096), ( 66$, $6075)\}$;
\item[(iv)] $G={\rm H}_3\cong \Cyc{2}\times \Alt{5}$ and 
\begin{itemize}
\item[$p=3$:] $(m,\chi_m(1))\in \{(4, 4),( 5, 5)\}$ or
\item[$p=5$:] $(m,\chi_m(1))=(4 , 4)$.
\end{itemize}
\end{itemize}
In all these cases we have $c_{\chi_m}=\pm 1$.
\end{proposition}

We now consider the four infinite families ${\rm A}_n$, ${\rm B}_n$, ${\rm D}_n$ and ${\rm I}_2(n)$.
Groups of type ${\rm A}_n$ are isomorphic to the symmetric group $\Sym{n+1}$.
The irreducible characters of $\Sym{n+1}$ are parametrized by the partitions $\lambda$ of $n+1$, see \cite{JK}. 
From \cite[Theorem 1.3, Lemma 3.1, Lemma 3.3 and Proposition 3.4]{PZ} we get the following result.

\begin{proposition}\label{pr:An}
Let $G$ be a group of type ${\rm A}_{n}$ with $n\geq 2$. 
Let $\chi_\lambda$ be the irreducible character of $G$ associated 
to the partition $\lambda$ of $(n+1)$. Suppose that  $\chi_\lambda$ is non-linear of non-zero 
$p$-defect. 
Then, $\chi$ is $p$-constant for some prime $p\leq n+1$ if, and only if, one of the following conditions holds:
\begin{itemize}
\item[(i)] $n=p-1\geq 2$ and $\lambda=(b,1^{p-b})$, where $2\leq b\leq p-1$;
\item[(ii)] $n=p\geq 3$ and $\lambda=(b+1,2,1^{p-b-2})$, where $1\leq b \leq p-2$;
\item[(iii)] $n+1=p+r\geq 5$, $2\leq r\leq  p$ and $\lambda$ is one of the following partitions
$$(p-a,r+1,1^ {a-1}),\;\; 1\leq a\leq p-r-1; \qquad (r,b,1^ {p-b}),\;\; 1\leq b \leq r.$$
\end{itemize}
In all these cases $c_{\chi_{\lambda}}=\pm 1$.
\end{proposition}

Applying Lemma \ref{linear}, it is easy to see that the linear character 
$\chi_{(1^n)}$ of $\Sym{n}$ is $p$-constant if and only if $n=p, p+1$.
\smallskip

Finite reflection groups of type ${\rm B}_{n}$, $n\geq 2$, are isomorphic to the wreath products 
$\Cyc{2}\wr 
\Sym{n}$. We then consider the more general case of wreath products $G=H\wr 
\Sym{n}=H^n\rtimes \Sym{n}$, where $H$ is any finite group. The classification of the 
irreducible $p$-constant 
 characters of $H\wr 
\Sym{n}$ 
(which is sometimes called the complete monomial group) shall be obtained 
applying the combinatorial version of the Murnaghan-Nakayama rule given in \cite{Bn}. We recall 
here some of the basic properties of $G$ with respect to conjugacy classes and irreducible 
characters. For further details, see \cite{Bn}.

We can think $G$ as the group of the $n\times n$ monomial matrices whose entries 
are elements of $H$.  
In particular, if $\pi$ is the $k$-cycle $(i_1,i_2,\ldots,i_k)$ of $\Sym{n}$, we 
denote by $(h_1i_1, h_2 i_2,\ldots , h_ki_k)$ the permutation matrix corresponding to the cycle
$\pi$ whose entries $1$ are replaced by the elements $h_1,\ldots,h_k\in H$. If 
$\mathcal{C}_1=\{1\},\mathcal{C}_2,\ldots,\mathcal{C}_r$ are the $r$ conjugacy classes of $H$, an
element  $(h_1 i_1,h_2 i_2,\ldots , h_k i_k)$ is said to be a $\mathcal{C}_j$-cycle if 
$h_k\cdots h_2h_1\in \mathcal{C}_j$.

If $\gamma=(\gamma_1,\ldots,\gamma_t)$ is a partition of $n$, we  write $\gamma\vdash n$ (or 
$|\gamma|=n$) and 
$\ell(\gamma)=t$ for the number of parts of $\gamma$. By $\vet{\gamma}$ we  denote a 
$r$-tuple $(\gamma^1,\ldots,\gamma^r)$ of partitions $\gamma^i$ such that
$\sum_i |\gamma^i|=n$ (we use $\vet{\gamma}\vdash n$ to emphasize the integer $n$). 
We also
allow that some of the components of $\vet{\gamma}$ can be the empty partition $\emptyset$.
Then, the set of conjugacy classes of $G$ is given by $\{C_{\vet{\gamma}} : 
\vet{\gamma}\vdash n\}$, 
where $C_{\vet{\gamma}}$ is the set of elements of $G$ whose $\mathcal{C}_j$-cycles have a cyclic 
decomposition $\gamma^j$, for all $j=1,\ldots,r$. Similarly, we have 
$\Irr{G}=\{\chi^{\vet{\gamma}}: 
\vet{\gamma} \vdash n\}$.
In particular, the irreducible characters of $G$ obtained by inflation from $\Irr{\Sym{n}}$ are the 
characters $\chi^{\vet{\gamma}}$, where $\vet{\gamma}=(\gamma,\emptyset,\ldots,\emptyset)$ with 
$\gamma\vdash n $ (so, $1_G=\chi^{((n),\emptyset,\ldots,\emptyset)}$).

The Murnaghan-Nakayama rule gives the value $\chi_{\vet{\delta}}^{\vet{\chi}}$ of the irreducible 
character $\chi^{\vet{\gamma}}$ on the conjugacy  class $C_{\vet{\delta}}$. Following \cite{Bn}, 
given $\vet{\gamma}\vdash n$, we denote by $\star\vet{\gamma}=\gamma^1\star\ldots\star \gamma^r$  
the skew 
shape formed by placing the Ferrers diagrams of 
$\gamma^1,\ldots,\gamma^r$ corner to corner (the Ferrers diagram of a partition $\gamma$ is 
constructed 
taking $\gamma_1$ boxes in the first row from the top, $\gamma_2$ boxes in the second row and so 
on).
Let $T$ be a $\star$-rim hook tableau of shape $\vet{\gamma}$ and type $\vet{\delta}$, i.e. a rim 
hook tableaux of shape $\star \vet{\gamma}$,  where the lengths of the rim hooks are found in the 
parts of the components of $\vet{\delta}$.
Let $\eta^{j,i}(T)$ be the partition formed by the lengths of the rim hooks of $\delta^j$ placed in 
the Ferrers diagram of $\gamma^i$ in $T$.
Set 
$$\star\omega(T)=\mathrm{sgn}(T) \prod_{i,j}^ r  (\psi_j^i)^{\ell(\eta^{j,i}(T))} ,$$
where $\psi_j^i$ is the value of the $i$-th irreducible character of $H$ on the class 
$\mathcal{C}_j$ (taking $\psi^1=1_H$)
and $\mathrm{sgn}(T)$ is the product of the signs of the rim hooks in $T$ (i.e., $(-1)^{t-1}$ where 
$t$ is 
the number of rows occupied by the rim hook).
Then 
$$\chi^{\vet{\gamma}}_{\vet{\delta}}=\sum \star \omega(T),$$
where the sum runs over all 
$\star$-rim hook tableaux of shape $\vet{\gamma}$ 
and type 
$\vet{\delta}$.

\begin{theorem}\label{th:wr}
Let $H$ be a finite group. An irreducible character $\chi$ of $H\wr \Sym{n}$ of 
non-zero $p$-defect is $p$-constant if and only if $p$ does not divide $|H|$, 
$H^n\leq \ker(\chi)$ and the corresponding character $\bar \chi$ of $\Sym{n}$ is 
$p$-constant.
\end{theorem}

\begin{proof}
Let $\chi=\chi^{\vet{\gamma}}\in \varUpsilon_p(G)$ for 
some prime $p$ dividing the order of $G=H\wr \Sym{n}$.
We want to prove that $p\leq n$ and  $\vet{\gamma}=(\gamma,\emptyset,\ldots,\emptyset)$ for some 
$\gamma \vdash n$. 

Assume first that $p$ divides the order of $H$. 
For any conjugacy class $\mathcal{C}_j$  of $H$ contained in $\varSigma_p(H)$, 
consider the conjugacy classes $C_{\vet{\delta_j}}, C_{\vet{\eta_{j,k}}}$ of $G$, where 
$\vet{\delta_j}$ is the $r$-tuple whose unique non-empty component is the partition $(n)$ at 
$j$-th position
and $\vet{\eta_{j,k}}$ is the $r$-tuple  having only two non-empty components: the partition 
$(n-1)$ at $j$-th 
position and the partition $(1)$ at $k$-th  position ($j\neq k$). 
Clearly, $C_{\vet{\delta_j}}, C_{\vet{\eta_{j,k}}} ,\subseteq \varSigma_p(G)$.
Since $\chi^{\vet\gamma}$ is non-zero on $\vet{\delta_j}$, then $\vet{\gamma}$ has 
a 
unique non-empty component, say at position $i$: $\vet{\gamma}=\left( 
\emptyset,\ldots,\emptyset,\gamma,\emptyset,\ldots,\emptyset\right)$,
where $\gamma\vdash n$ must be a hook,  i.e. $\gamma=(1^a, n-a)$ for some $0\leq a\leq n-1$.
Furthermore, $\chi^{\vet\gamma}_{\vet{\delta_j}}=(-1)^a \psi^i_j$ and so the character
$\psi^i$ is a $p$-constant character of $H$. 

Now, suppose $1\leq a\leq n-2$. Then 
$$\chi_{\vet{\eta_{j,1}}}^{\vet{\gamma}}=(-1)^{a-1} \psi_1^i\psi_j^i +(-1)^a\psi_1^i\psi_j^i=0,$$ 
a contradiction.
If $\gamma=(1^n)$, then  
$\chi_{\vet{\delta_j}}^{\vet{\gamma}}=(-1)^{n-1}\psi_j^i=\chi_{\vet{\eta_{j,1}}}^{\vet{\gamma}}
=(-1)^ { n-2 }
\psi_1^i\psi_j^i$ implies that $\psi_1^i=-1$ is the degree of $\psi^i$, an absurd.
So, $\gamma=(n)$. We get 
$\chi_{\vet{\delta_j}}^{\vet{\gamma}}=\psi_j^i=\psi_1^i\psi_j^i$ and so $\psi^i$ is linear.
Now,  $\chi_{\vet{\mu_{j,k}}}^{\vet{\gamma}}=
\psi_k^i \psi_j^i$ implies $\psi_k^i=1$ for all $k\neq j$.
It follows that either $\psi^i=1_H$ (and so $\chi=1_G$) or $p=2$ and $G=\Cyc{2}\wr 
\Sym{n}$. 

Hence, we may now assume $p\leq n$. Our next goal is to prove that $\vet{\gamma}$ has a 
unique non-empty component. The choice of the conjugacy classes $C_{\vet{\delta}}$ defined below is 
functional also for proving Proposition \ref{pr:Dn}.
Assume that $\vet{\gamma}$ has at least two non-empty components and write $n=pk+t$ with $0\leq t<p$.
If $t=0$, take $\vet{\alpha}=((pk), \emptyset,\ldots,\emptyset)$. Then $C_{\vet{\alpha}}\subseteq 
\varSigma_p(G)$, but $\chi^{\vet{\gamma}}_{\vet{\alpha}}=0$, a contradiction.
It $t>0$, take the classes of $p$-singular elements of $G$ corresponding to the following $r$-tuples:
$$\vet{\alpha_1}=\left((t,pk),\emptyset,\ldots,\emptyset\right), 
\vet{\alpha_2}=\left( (t), (pk),\emptyset,\ldots,\emptyset\right),$$
$$\vet{\alpha_3}=\left( (t),\emptyset, 
(pk),\emptyset,\ldots,\emptyset\right),\ldots,
\vet{\alpha_r}=\left( (t),\emptyset,\ldots,\emptyset,(pk)\right);$$
$$\vet{\beta_2}=\left(\emptyset,(t,pk),\emptyset,\ldots,\emptyset\right), 
\vet{\beta_3}=\left( 
\emptyset,\emptyset,(t,pk),\emptyset,\ldots,\emptyset\right),\ldots,\vet{\beta_r}=\left( 
\emptyset,\ldots,\emptyset,(t,pk)\right).$$
From  $\chi^{\vet{\gamma}}_{\vet{\alpha_1}}\neq 0$ it follows that $\vet{\gamma}$ has at most  
two non-empty components. Then $\vet{\gamma}$ has exactly two non-empty components, say  in position 
$i$ and $j$,  that must be hooks of length, respectively,
$pk$ and $t$.
For any $s=2,\ldots,r$ we have
$$\chi^{\vet{\gamma}}_{\vet{\alpha_1}}=\epsilon \psi_1^i \psi_1^j= 
\chi^{\vet{\gamma}}_{\vet{\alpha_s}}=\epsilon \psi_s^i\psi_1^j\neq 0,$$
for some $\epsilon=\pm 1$, and so $\psi^i=1_H$ (i.e. $i=1$).
Now, 
$$\chi^{\vet{\gamma}}_{\vet{\alpha_1}}=\epsilon \psi_1^j= 
\chi^{\vet{\gamma}}_{\vet{\beta_s}}=\epsilon  \psi_s^j$$
for all $s=2,\ldots,r$ also implies $\psi^j=1_H$ and $j=1$, a contradiction. 

We conclude that  $\vet{\gamma}$ has a unique non-empty component, $\gamma\vdash n$ 
at the $i$-th position. We prove now that actually $i=1$.
For all $s=2,\ldots,r$ we have
$$\chi^{\vet{\gamma}}_{\vet{\alpha_1}}=\nu \psi_1^i \psi_1^i= 
\chi^{\vet{\gamma}}_{\vet{\alpha_s}}=\nu \psi_s^i\psi_1^i\neq 0,$$
for some $\nu\in \{\pm 1,\pm 2\}$,
which implies $\psi^i=1_H$ and so 
$\vet{\gamma}=(\gamma,\emptyset,\ldots,\emptyset)$ for some partition $\gamma$ of $n$.
As observed before, in this case $H^n\leq \ker(\chi)$ and the corresponding character $\bar \chi$
must be a $p$-constant character of $\Sym{n}$.

We are left to consider the case when $p=2$ divides $|H|$ and $\bar \chi$ is a $2$-constant character of $\Sym{n}$: this happens only when 
$n=2,3$ and $\bar \chi$ is linear.
However,  $G=\Cyc{2}\wr \Sym{2}$ is a $2$-group and by  Lemma \ref{p-group}, $\chi=1_G$.
Also $G=\Cyc{2}\wr \Sym{3}$ is isomorphic to $\Cyc{2}\times \Sym{4}$ and, by Proposition 
\ref{direct}, we get $\chi=1_G$.
\end{proof}

\begin{cor}\label{cor:Bn}
Let $G$ be a group of type ${\rm B}_{n}$ with $n\geq 2$. 
Let $\chi$ be an irreducible character of $G$  of non-zero 
$p$-defect. 
Then, $\chi$ is $p$-constant for some prime $p\leq n$ if and only if $p>2$, $\Cyc{2}^n\leq 
\ker(\chi)$ and the corresponding character $\bar \chi$ of $\Sym{n}$ is $p$-constant.
\end{cor}

The groups ${\rm D}_{n}$ are subgroups of index $2$ in ${\rm B}_{n}$.
Let $\alpha,\beta$ be partitions of some $0\leq t\leq n$ such that $|\alpha|+|\beta|=n$ and 
let $\chi^{(\alpha,\beta)}$ be the corresponding irreducible character of $H=\Cyc{2}\wr \Sym{n}$.
If $\alpha\neq \beta$, the characters $\chi^{(\alpha,\beta)}$ and $\chi^{(\beta,\alpha)}$ restrict to the same irreducible character of $G={\rm D}_{n}$. The character $\chi^{(\alpha,\alpha)}$ splits on $G$ as the sum of two 
irreducible characters $\chi^{(\alpha,\alpha)}_+, \chi^{(\alpha,\alpha)}_-$ of $G$.
Furthermore, a class $C_{(\alpha,\beta)}$ of $H$ belongs to $G$ only if $\beta$ has even 
length and
a class $C_{(\alpha,\beta)}$ of $H$ splits into two classes of $G$ if and only if 
$\beta=\emptyset$ and the parts of $\alpha$ are all even (see \cite{Pf}).

\begin{proposition}\label{pr:Dn}
Let $G$ be a group of type ${\rm D}_{n}$ with $n\geq 4$. 
Let $\chi$ be an irreducible character of $G$ of non-zero 
$p$-defect. 
Then, $\chi$ is $p$-constant for some prime $p\leq n$ if and only if $\chi$ is the restriction to $G$ of an 
irreducible $p$-constant character of ${\rm B}_{n}$.
\end{proposition}

\begin{proof}
Suppose that $\chi\in \varUpsilon_p(G)$. Our 
first goal is to show that $\chi$ is the restriction to $G$ of a character 
$\chi^{(\alpha,\emptyset)}$ of $H={\rm B}_{n}$.

So, let $\chi$ be the restriction to $G$ of some character $\chi^{(\alpha,\beta)}$ (or one of the 
two characters $\chi^{(\alpha,\alpha)}_{\pm}$).
Let $n=pk+t$ with $0\leq t< p$. If $pk$ and $t$ are not both even, we can prove that 
$\beta=\emptyset$, arguing as as done proving Theorem \ref{th:wr}, since the 
classes $C_{((t,pk),\emptyset)}$ and $C_{(\emptyset,(t,pk))}$ belong to $G$.

Suppose now that $\beta\neq \emptyset$ and that both $pk$ and $t$ are even. 
If $t\geq 2$, we consider the classes $C_{\vet{\eta_1}},C_{\vet{\eta_2}}$ of $G$, where 
$\vet{\eta_1}=(\emptyset,(t,pk))$ and $\vet{\eta_2}=((pk),(1,t-1))$.
From $\chi^{(\alpha,\beta)}_{\vet{\eta_1}}\neq 0$ we get that $\alpha,\beta$ are both 
hooks of length, respectively,  $pk$ and $t$ and from $\chi^{(\alpha,\beta)}_{\vet{\eta_2}}\neq 0$ 
we get $\beta=(t)$ or $(1^t)$. In this case, we have
$\chi^{(\alpha,\beta)}_{\vet{\eta_1}}=-\chi^{(\alpha,\beta)}_{\vet{\eta_2}}$, a contradiction.

Now, let $t=0$ and $n=pk\geq 4$ be even (whence $k\geq 2$).
Consider the classes $C_{\vet{\delta_1}},C_{\vet{\delta_2}}$, where 
$\vet{\delta_1}=(\emptyset,(p,pk-p))$ and $\vet{\delta_2}=(\emptyset, (n/2,n/2))$. Observe that 
$\vet{\delta_1}=\vet{\delta_2}$ if, and only if, $n=2p$.
From $\chi^{(\alpha,\beta)}_{\vet{\delta_1}}\neq 0$ we 
get that $\alpha,\beta$ are both hooks of length, respectively,  $pk-p$ and $p$. However in this 
case, $\chi^{(\alpha,\beta)}_{\vet{\delta_2}}= 0$, unless $n=2p$. 
So, suppose $n=2p$ and take the class $C_{\vet{\delta_3}}$, where 
$\vet{\delta_3}=((1,p-1,p),\emptyset)$:
$\chi^{(\alpha,\beta)}_{\vet{\delta_3}}= 0$, unless $\beta=(p)$ or $(1^p)$. In this case, 
comparison between $\chi^{(\alpha,\beta)}_{\vet{\delta_2}}$ 
and $\chi^{(\alpha,\beta)}_{\vet{\delta_3}}$ always produces an absurd.

We conclude that  $\chi$ is the irreducible restriction to $G$ of $\chi^{(\alpha,\emptyset)}$ 
for some $\alpha \vdash n$. 
As we already remarked, the characters $\chi^{(\alpha,\emptyset)}$  are obtained by inflation from $\Sym{n}$. 
Hence, we are left to exclude that the character $\chi^{(\alpha,\emptyset)}|_G$ is $p$-constant, 
even if the corresponding character $\chi_\alpha$ of $\Sym{n}$ is not $p$-constant.
In particular we have to consider  suitable conjugacy classes $C_{\vet{(\alpha,\emptyset)}}$ of $p$-singular elements of $G$, such that $\alpha\vdash n$ does not consist of even parts. 
However, this was already done in the original proof of Proposition \ref{pr:An} in \cite{PZ} and 
\cite{end}, where the conjugacy classes that were analyzed never consist of even parts.
\end{proof}

Finally, we consider the  dihedral groups ${\rm I}_2(n)$  of order $2n$. 
A direct check of their character table, that can be found, for instance, in \cite{JL}, gives the following result.

\begin{proposition}\label{die}
Let $G$ be a group of type ${\rm I}_2(n)$ with $n\geq 3$ and let $\chi$ be a non-linear irreducible character of $G$ of non-zero $p$-defect. 
Then, $\chi$ is $p$-constant for some prime $p$ dividing $2n$ if and only if $p=3=|G|_3$
and $\chi=\psi_{n/3}$, in the notation of \cite[page 182]{JL}. In particular, $c_\chi=-1$.
\end{proposition}

\begin{proof}[Proof of Theorem \ref{th1}]
Let $G$ be a finite reflection group. Suppose that $1_G\neq \chi \in \varUpsilon_p(G)$, for some prime $p$ dividing $|G|$. Since  $G$ can  be written as direct product of irreducible reflection groups $G=H_1\times\ldots\times H_r$, by Proposition \ref{direct}, $p$ divides $|H_i|$ for a unique value of $i\in \{1,\ldots,r\}$. Furthermore, $\chi=\psi \otimes 1_{K}$, where $\psi \in \varUpsilon_p(H_i)$, $K=\prod_{j\neq i} H_j$ and $c_\chi=c_\psi$.

Now, by Propositions \ref{ecc} and \ref{pr:An},  $c_{\psi}=\pm 1$ if $H_i$ is of exceptional type or $H_i=\Sym{n}$ for $n\geq 3$. If $H_i$ is of type ${\rm B}_n$, then $\psi$ is obtained by inflation from a $p$-constant character of $\Sym{n}$ (Corollary \ref{cor:Bn}) and so, also in this case, $c_\psi=\pm 1$. Clearly, $c_\psi=\pm 1$ even when $H_i$ is of type ${\rm D}_n$ (Proposition \ref{pr:Dn}).
Finally, if $H_i$ is a dihedral group, by Proposition \ref{die} we have $c_\psi=-1$.
\end{proof}

\begin{proof}[Proof of Theorem \ref{th2}]
As done proving Theorem \ref{th1}, write $G$ as direct product of irreducible reflection groups $G=H_1\times\ldots\times H_r$. Since $p=2$ divides the order of any $H_i$, by Proposition \ref{direct}, $r=1$. Now, by Propositions \ref{direct}, \ref{pr} and \ref{pr:Dn} and Corollary \ref{cor:Bn}, $G$ must be a dihedral group $I_2(n)$ where $n\geq 1$ ($I_2(1)\cong \Cyc{2}$).
Now, by Proposition \ref{die} we obtain that $\chi$ must be a linear character of $G$.
Finally, by Proposition \ref{linear}, $G$ has order $2n$, where $n$ is odd. The statement follows from the fact that, 
when $n$ is odd,  $I_2(n)$ has only two linear characters.
\end{proof}

\section{Open problems}\label{Sec4}

Our final aim is to characterize the finite groups for which $\varUpsilon_p(G)\neq \{1_G\}$. From Propositions \ref{linear} and \ref{direct} and Lemma \ref{lm:center}, we are 
reduced to consider non-linear characters of  indecomposable groups with trivial center.

The condition  $c_\chi=\pm 1$ seems to be too weak for a possible classification. Namely, by Theorem \ref{th1}, the $p$-constant characters (of non-zero $p$-defect) of all finite reflection groups have this property.
On the other hand, one may focus on 
$p$-constant characters $\chi$ such that $|c_\chi|>1$. Using \cite[\texttt{SmallGroups}]{GAP}, the following conjecture has been verified for all groups of size less than $1250$:

\begin{conj}\label{conj1}
Let $G$ be a finite group which admits an irreducible $p$-constant character $\chi$ such that 
$|c_\chi|>1$. Then $p\neq 2$ and the Sylow $p$-subgroups of $G$ are homocyclic.
\end{conj}

Notice that when $c_\chi=\pm 1$ we can observe a different behavior. For instance, the 
irreducible character $\chi$ of $\Alt{6}$ of degree  $9$ is $2$-constant with $c_\chi=+1$
and the irreducible character of $\PSL_3(2)$ of degree $7$ is $2$-constant with 
$c_\chi=-1$. However, in both cases the Sylow $2$-subgroups are dihedral of order $8$.

The smallest group  that admits an irreducible $p$-constant character with $|c_\chi|>1$ is obtained taking $p=3$ in the following:

\begin{lemma}\label{Q8}
For any odd prime $p$, there exists a finite group $G$ admitting an irreducible $p$-constant character $\chi$ with $c_\chi=+2$.
\end{lemma}

\begin{proof}
For any prime $p$, take a $2$-dimensional space $V$ defined over the field $\F_p$.
The group $\Q_8=\{\pm 1, \pm i, \pm j , \pm k\}$ admits an irreducible representation $\Phi$ with space of representation $V$, given by
$$\Phi(i)=\begin{pmatrix} a & b \\ b & -a   \end{pmatrix}, \quad 
\Phi(j)=\begin{pmatrix} 0 & -1 \\ 1 & 0   \end{pmatrix},  
$$
where $a,b\in \F_p$ are chosen such that $a^2+b^2+1=0$.
Define
$$G(p)=\left \{\begin{pmatrix} 1 & 0 \\ v & \Phi(g) \end{pmatrix}: v \in V, g \in \Q_8\right \}\cong V\rtimes \Q_8.$$
The group $\Q_8$ acts on $V$ as a group of fixed-point free automorphisms and so $G(p)$ is a Frobenius group.
The elements of $G(p)$ have order $\{1,2,4,p\}$. The irreducible character $\chi$ of $G(p)$ of degree $2$,  obtained by inflation from $\Q_8$,  is $p$-constant with $c_\chi=+2$.
\end{proof}

The smallest group that satisfies Conjecture \ref{conj1} with a Sylow $p$-subgroup which is not elementary abelian is $G_{248,253}$ whose structure is $(\Cyc{9}\times \Cyc{9}) \rtimes\Q_8$, while
the smallest one with a Sylow $p$-subgroup of rank $>2$ is $G_{248,730}$ whose structure is $(\Cyc{3}\times 
\Cyc{3}\times\Cyc{3}\times \Cyc{3}) \rtimes\Q_8$ (here $G_{n,k}$ denotes the $k$-th group of size $n$, according to 
\cite{GAP}).
In both cases, $p=3$ and $\chi(1)=c_\chi=+2$. 

The smallest group admitting an irreducible $p$-constant character with  $c_\chi<-1$ is ${\rm M}_{10}$ ($c_\chi=-2$ and $\chi(1)=16$), while the smallest perfect group with $|c_\chi|>1$ is the following.

\begin{example}
Consider the  group $H=\langle h_1,h_2,h_3 \rangle\leq \SL_2(11)$, where
$$h_1=\begin{pmatrix} 3 & 1 \\ 1 & -3  \end{pmatrix},\quad
h_2=\begin{pmatrix} 0 & 1 \\ -1 & 0  \end{pmatrix},\quad
h_3=\begin{pmatrix} 2 & 0 \\ 0 & 6  \end{pmatrix}.
$$
The group  $H$ is isomorphic to $\SL_2(5)$ and acts regularly on $V=\F_{11}^2$, hence the group 
$G=V\rtimes H$ is a Frobenius group with a single class of $11$-singular elements.
We get that $\varUpsilon_{11}(G)=\Irr{G}$.
Also, for $\chi(1)=1,2,3,4,5,6$ we obtain $c_\chi=\chi(1)$ and for $\chi(1)=120$ we get $c_\chi=-1$.
Clearly the Sylow $11$-subgroups of $G$ are elementary abelian of order $11^2$.
\end{example}

Looking at the perfect groups, we can state a stronger version of Conjecture \ref{conj1}. 
Indeed, a direct check on the perfect groups of size $<15360$, whose character tables are contained in the package 
\texttt{AtlasRep 1.5.0}, seems to confirm the following.

\begin{conj}\label{perfect}
Let $G$ be a finite perfect group admitting an irreducible $p$-constant character $\chi$ such that 
$|c_\chi|>1$. Then $p=3$, $c_\chi=\pm 2$ and the Sylow $3$-subgroups of $G$ are isomorphic to 
$\Cyc{3}\times \Cyc{3}$.
\end{conj}

Notice that the only simple sporadic group admitting irreducible $p$-constant characters with $c_\chi=\pm 2$ is the 
Mathieu group ${\rm M}_{22}$, that satisfies Conjecture  \ref{perfect}.
The other known simple groups satisfying Conjecture \ref{perfect} are the groups $\PSL_3(q)$ when $\gcd(q-1,9)=3$ and $\PSU_3(q)$ when $\gcd(q+1,9)=3$. 
Indeed, in these cases, the Sylow $3$-subgroups are isomorphic to $\Cyc{3}\times\Cyc{3}$ and using the character tables described in \cite{SF} we may prove that following.

\begin{proposition}\label{pr:L3}
Let $G\in \{\PSL_3(q),\PSU_3(q)\}$ and let $r$ be a prime dividing $|G|$ such that $\gcd(r,q)=1$.
Let  $\delta=+1$ if $G=\PSL_3(q)$, $\delta=-1$ if $G=\PSU_3(q)$. Set $d=\gcd(q-\delta,3)$. 
Then $1_G\neq \chi \in \varUpsilon_r(G)$ if and only if one of the following holds:
\begin{itemize}
\item[(i)] $r$ divides $(q^2+\delta q+1)/d$ and
\begin{itemize}
\item[(a)] $\chi(1)=q^2+\delta q$: in this case $c_\chi=-1$;
\item[(b)] $\chi(1)=q^3$: in this case  $c_\chi=\delta$.
\end{itemize}

\item[(ii)] $2\neq r$ divides $(q+\delta)$ and
\begin{itemize}
\item[(a)] $\chi(1)=q^3$: in this case $c_\chi=-\delta$;
\item[(b)] $d=1$,  $\gcd(q+\delta,9)=r=3$ and $\chi(1)=q^3-\delta$: in this case $c_\chi=\delta$.
\end{itemize}

\item[(iii)] $d=r=\gcd(q-\delta,9)=3$ and 
\begin{itemize}
\item[(a)] $\chi(1)=q^2+\delta q$: in this case $c_\chi=+2$;
\item[(b)] $\chi(1)=q^3$: in this case $c_\chi=\delta$;
\item[(c)] $\chi(1)=\frac{(q+\delta)(q^2+\delta q+1)}{3}$: in this case $c_\chi=-\delta$.
\end{itemize}
\end{itemize}
\end{proposition}

\end{document}